\documentclass[11pt,twoside]{article}

\usepackage{mathrsfs,amsfonts,amsmath,amssymb}
\usepackage{url}

\usepackage{comment} 
\usepackage{color}
\usepackage{latexsym}
\usepackage{enumitem}

\newtheorem{satz}{Theorem}
\newtheorem{proposition}[satz]{Proposition}
\newtheorem{theorem}[satz]{Theorem}
\newtheorem{lemma}[satz]{Lemma}

\def\Z{\mathbb {Z}}
\def\F{\mathbb {F}}
\def\E{\mathsf{E}}

\def\a{\alpha}

\def\o{\omega}
\def\({\big (}
\def\){\big )}

\def\dim{{\rm dim}}
\def\le{\leqslant}
\def\ge{\geqslant}
\def\_phi{\varphi}
\def\eps{\varepsilon}

\def\Gr{{\mathbf G}}

\def\la{\lambda}

\def\F{\mathbb {F}}

\def\bp{\bigskip}

\author{Shkredov I.D.}
\title{On common energies and sumsets  II 
}
\date{}
\begin{document}
	\maketitle

\begin{center}
	Annotation.
\end{center}

{\it \small
    We continue to study the relationship  between the size of the sum of a set and the common energy of its subsets. 
    We find a rather sharp subexponential dependence between the doubling constant of a set $A$ and the minimal common energy taken over all  partitions  of $A$ into two disjoint subsets. 
    As an application, we give a proof of the well--known arithmetic regularity lemma with better dependence  on parameters. 
}
\\

\section{Introduction}
\label{sec:intr}

Given an abelian group $\Gr$ and two sets $A,B\subseteq \Gr$, define  
the {\it sumset} of $A$ and $B$ as 
\begin{equation}\label{def:A+B_intr}
    A+B:=\{a+b ~:~ a\in{A},\,b\in{B}\}\,.
\end{equation} 
The study of the structure of sumsets is a fundamental  problem in  classical additive combinatorics see, e.g., \cite{Freiman_book}, \cite{TV}, \cite{GGMT_Marton_bounded} and other papers.
Another important combinatorial concept closely related to the sumset $A+A$
 is the {\it additive energy} $\E(A,A)$ or, more generally, the {\it common additive energy} $\E(A,B)$ of $A$ and $B$, which is defined as
\begin{equation}\label{def:common_energy_intr}
 \E^{} (A,B) = |\{ (a_1,a_2,b_1,b_2) \in A\times A \times B \times B ~:~ a_1 - b^{}_1 = a_2 - b^{}_2 \}| \,.
\end{equation}
The additive energy and the size of the sumset are trivially related  via 
the Cauchy--Schwarz inequality, namely, 
\begin{equation}\label{f:energy_CS}
    \E^{} (A,B) |A \pm B| \ge |A|^2 |B|^2 \,,
\end{equation}
but a much deeper connection is given to us by the famous Balog--Szemer\'edi--Gowers theorem (see \cite{BSz_statistical} and \cite{Gowers_4}). 

\begin{theorem}
    Let $\Gr$ be an abelian group, $A\subseteq \Gr$ be a set and $K\ge 1$ be a real number. 
    Then 
\begin{equation}\label{f:BSzG_intr}
    \E(A,A) \gg \frac{|A|^3}{K^{C_1}} \quad \quad 
    \mbox{ iff }
        \quad \quad 
    \exists A' \subseteq A ~:~ |A'+A'| \ll K^{C_2} |A| 
    \quad 
    \mbox{ and }
    \quad
    |A'| \gg \frac{|A|}{K^{C_2}} \,.
\end{equation}
    Here $C_1 = O(C_2)$ and $C_2 = O(C_1)$. 
\label{t:BSzG_intr}
\end{theorem}

    In recent paper \cite{sh_common_energy} the author obtained another polynomial relation between the additive energy and $|A+A|$.



\begin{theorem}
    Let $\Gr$ be an abelian group, $A\subseteq \Gr$ be a set and $K\ge 1$ be a real number.
    Then 
\begin{equation}\label{eq:intr}
    |A+A| \ll K^{C_1} |A| 
        \quad \quad 
    \mbox{ iff }
        \quad \quad 
    \forall X,Y \subseteq A,\, 
    |X|\ge |A|/2 ~:~ \E(X,Y) \ge 
    \frac{|X| |Y|^2}{K^{C_2}} 
    \,. 
\end{equation}
     Here, as above, $C_1 = O(C_2)$ and $C_2 = O(C_1)$.
\label{t:sh_pre_intr}
\end{theorem}

Therefore it is possible to express the doubling constant of $A$ in terms of the common additive energy  of subsets of $A$. 

In this paper we consider a weaker condition than \eqref{eq:intr}, namely, instead of taking all pairs $X,Y \subseteq A$, $|X|\ge |A|/2$ we have deal with $X,Y \subseteq A$ such that $X \cup Y = A$ and the union is disjoint (then, trivially, either $|X|$ or $|Y|$ is at least $|A|/2$). 
These kinds of restrictions were considered in excellent paper \cite{ER_connectedness} (at the graph level), although the common energy $\E(X,Y)$ is a nonlinear and more subtle quantity than the usual graph density. 
It turns out that in the case of disjoint unions we do {\it not} have polynomial dependence between this new quantity (the rigorous definition can be found in Section \ref{sec:def}) and 
the {\it doubling constant} $\mathcal{D}[A]:= |A+A|/|A|$ of $A$. 
More precisely (see Theorems \ref{t:main}, \ref{t:constr} below), we obtain

\begin{theorem}
    Let $\Gr$ be an abelian group, $A\subseteq \Gr$ be a set.
    Suppose that for any $X,Y \subseteq A$ such that $X \cup Y = A$ and $X\cap Y = \emptyset$ one has 
\begin{equation}\label{f:main_intr_1}
    \E (X,Y) \ge \frac{|X|^2 |Y|^2}{K |A|} \,.
\end{equation}
    Then 
\begin{equation}\label{f:main_intr_2}
    \frac{|A+A|}{|A|} \le \exp (O(K^{1/3} \log^2 K)) \,.
\end{equation}
    On the other hand, there is an abelian group $\Gr$ and a set $A\subseteq \Gr$ such that \eqref{f:main_intr_1} takes place 
    for all $X,Y \subseteq A$ such that $X \cup Y = A$ and $X\cap Y = \emptyset$
    but 
\begin{equation}\label{f:main_intr_3}
    \frac{|A+A|}{|A|} \ge  \exp ( \Omega (K^{1/4} \log K)) \,.
\end{equation}
\label{t:main_intr}
\end{theorem}

If $|A+A| = \mathcal{D}[A] \cdot |A|$, then the Cauchy--Schwarz inequality \eqref{f:energy_CS} gives us $\E(X,Y) \ge \frac{|X|^2 |Y|^2}{\mathcal{D}[A] |A|}$ for {\it any} $X,Y \subseteq A$ and hence  estimates \eqref{f:main_intr_2}, \eqref{f:main_intr_3} tell us that the dependence between the quantity $\mathcal{D}[A]$ and $K$ in \eqref{f:main_intr_1} 
is subexponential.

As an application of Theorem \ref{t:main_intr} we obtain a 
direct proof 
of the well--known arithmetic regularity lemma of Green and Sisask \cite[Proposition 3.2]{GS_max_AP3} with better 
dependence  
on parameters, see Section \ref{sec:application}. 
Nowadays, 
the arithmetic regularity lemma is a rather popular tool in additive combinatorics, some of its applications can be found in papers  \cite{ARL1}, \cite{ARL2}, \cite{GS_max_AP3} and \cite{ARL3}.

We hope that the tools developed in this paper will add flexibility to working with the doubling constants.

\section{Definitions}
\label{sec:def}

Below $\Gr$ denotes an abelian group with the group operation $+$. 
The sumset of two sets $A,B \subseteq \Gr$ was defined in formula \eqref{def:A+B_intr} of the introduction. 
In a similar way one can define the {\it difference sets} and the {\it higher sumsets}, e.g., $2A-A$ is $A+A-A$. 
Recall that the {doubling constant} of a finite set $A$ is 
given 
by the formula 
\begin{equation}\label{def:doubling}
    \mathcal{D} [A] := \frac{|A+A|}{|A|} \,.
\end{equation}
We say that the sum of $A$ and $B$ is direct 
if $|A+B|= |A||B|$.
In this case we sometimes write $A\dotplus B$. 
The important Pl\"unnecke--Ruzsa inequality (see, e.g., \cite{Ruzsa_Plun} or \cite{TV}) says that for any positive integers $n$ and $m$ 
the following holds 
\begin{equation}\label{f:Pl-R} 
    |nA-mA| \le \left( \frac{|A+A|}{|A|} \right)^{n+m} \cdot |A| \,.
\end{equation} 
Therefore, 
bound \eqref{f:Pl-R} connects the cardinality of the 
original 
set $A$ and its higher sumsets $nA-mA$. 
The common  additive energy of two sets was defined in \eqref{def:common_energy_intr} and the generalization of the common additive energy for four sets 
$A,B,C,D \subseteq \Gr$ is given by the formula 
\[
    \E(A,B,C,D) = \{ |\{ (a,b,c,d) \in A\times B \times C \times D ~:~ a+b= c+d \}| \,. 
\]
The disjoint union of two sets $A,B \subseteq \Gr$ is 
denoted 
as $A\bigsqcup B$.

Given finite sets $A, B\subseteq \Gr$ and a real number $T\ge 1$
 we defined the quantity $\mathcal{E}_T [A]$ in paper \cite{sh_common_energy} as 
\[
     \mathcal{E}_T [A] = \max_{X\subseteq A,\, |X|\ge |A|/T,\, Y \subseteq A}\, \frac{|X|^2 |Y|^2}{|A|\E(X,Y)} \,.
\]
    Thus Theorem \ref{t:sh_pre_intr} from the introduction says that the quantities $\mathcal{D}[A]$ and $\mathcal{E}_2 [A]$ are polynomially equivalent. 
    In this paper we consider the new quantity 
\[
    \mathcal{E}_* [A] = \max_{X \bigsqcup Y = A}\, \frac{|X|^2 |Y|^2}{|A|\E(X,Y)} \,.
\]
    In view of inequality \eqref{f:energy_CS} and the discussion after Theorem \ref{t:sh_pre_intr} 
    one has  for any $T\ge 2$ that 
\[
    \mathcal{E}_* [A] \le \mathcal{E}_2 [A] \le \mathcal{E}_T [A] \le \mathcal{D}[A] \,.
\]

The signs $\ll$ and $\gg$ are the usual Vinogradov symbols. 
If $a\ll b$ and $b\ll a$, then we write $a\sim b$. 
When the constants in the signs  depend on a parameter $M$, we write $\ll_M$ and $\gg_M$.
Let us denote by $[n]$ the set $\{1,2,\dots, n\}$.
All logarithms are to base $e$.
For a prime number $p$ we write $\F_p = \Z/p\Z$.



\section{The proof of the main result}
\label{sec:proof}

In this section we obtain the first part of Theorem \ref{t:main_intr} from the introduction. 
%
%
We start with a simple lemma.

\begin{lemma}
    Let $A,B\subseteq \Gr$, $|A|\ge |B|$ and 
    $\E(A,B) = |A||B|^2/K$.
    Then there is $X\subseteq \Gr$ such that 
\begin{equation}\label{f:E(A,B)_Schoen_X}
    \frac{|A|}{4K|B|} \le |X| \le \frac{2K|A|}{|B|} \,,
\end{equation}
    and 
\begin{equation}\label{f:E(A,B)_Schoen_B+X}
    |A \cap (B+X)|\ge \frac{|A|}{4K} \,.
\end{equation}
\label{l:E(A,B)_Schoen}
\end{lemma}
\begin{proof}
    We have $\E(A,B) = \sum_{x} |A\cap (B+x)|^2$ and hence there is $x\in \Gr$ such that 
\[
    |A\cap (B+x)| \ge \frac{\E(A,B)}{|A||B|} \ge \frac{|B|}{K} \,.
\]
    Also, for $A' = A\setminus (B+x)$ one has
\[
    \E(A',B) \ge \E(A,B) - 2 |A\cap (B+x)| |B|^2 \,.
\]
    Thus, iterating this procedure we obtain a set $X\subseteq \Gr$ such that $|X||B|/2K \le |A|$ and 
\[
    2 |A \cap (B+X)| |B|^2 \ge 2^{-1} \E(A,B) = |A||B|^2/(2K) \,.
\]
    It follows that $|A \cap (B+X)| \ge |A|/(4K)$ and the last inequality implies 
    $|X| \ge |A|/(4K|B|)$. 
This completes the proof.
$\hfill\Box$
\end{proof}

\bp

Now we show that sets $A\subseteq \Gr$ with small quantity $\mathcal{E}_* [A]$ have some rather specific properties.
Namely, Proposition \ref{p:A'_large} below shows that if such $A$ contains a large subset with small doubling constant, then $A$ also has small doubling constant. 
Of course, this result
does not hold 
for an arbitrary set. 

\begin{proposition}
    Let $A\subseteq \Gr$ be a set,  $M:=\mathcal{E}_* [A]$ and there is a subset $A' \subseteq A$ such that 
\begin{equation}\label{cond:A'_large}
    |A'+A'| \le K|A'| 
    \quad \quad 
        \mbox{ and }
    \quad \quad 
    |A'| \ge (1-\eps) |A| \,,
\end{equation}
    where 
    $\eps = (2^{6} M^{})^{-1}$.
    Then $|A+A| \ll K^4 M^2 |A|$. 
\label{p:A'_large}
\end{proposition}
\begin{proof}
    We can assume that $|A|$ is sufficiently large; otherwise, there is nothing to prove. 
    By the covering lemma (see, e.g.,  \cite[Section 2.4]{TV}) we find $Z \subseteq A$ such that $|A'+Z| = |A'||Z|$ and $A\subseteq A'-A'+Z$. 
    Let us remark that 
\[
    (1-\eps) |A| |Z| \le |A'||Z| = |A'+Z| \le |A+A| \le 2^{-1}(|A|^2 + |A|)
\]
    and therefore 
\[
    |Z| \le \frac{|A|+1}{2(1-\eps)} \le \frac{3|A|}{4} \,.
\]
    In particular,  we have $|Z^c| \ge |A|/4$, where $Z^c = A\setminus Z$. 
    Put $A'_Z  = A'\cap Z^c$ and $\Omega = Z^c \setminus A'_Z$.
    Clearly, $|\Omega| \le \eps |A|$. 
    By the definition of the quantity $\mathcal{E}_* [A]$ and the fact 
    that the sum $A'+Z$ is direct, we have 
\[
    \frac{|Z|^2 |Z^c|^2}{M |A|} \le \E(Z,Z^c) \le \E(Z,A'_Z) + \E(Z,Z,Z^c, \Omega) + \E(Z,Z,A'_Z,\Omega)
    \le \E(Z,A') + 2 |\Omega| |Z|^2 
\]
\[
    \le 
    |Z||A| + 2\eps |A| |Z|^2 \,.
\]
    In view of the 
    inequality 
    $|Z^c| \ge |A|/4$ and our choice of $\eps$ one has 
\[
    |Z| \le \frac{2M|A|^2}{|Z^c|^2} \le 2^5 M \,.
\]
    Using the last bound, the covering lemma  and the Pl\"unnecke inequality \eqref{f:Pl-R}, we derive
\[
    |A+A| \le |2A'-2A'| |Z+Z| \ll K^4 M^2 |A| \,. 
\]
This completes the proof.
$\hfill\Box$
\end{proof}

\bp 

Now we are ready to prove the main result of this section.

\begin{theorem}
    Let $A\subseteq \Gr$ be a set.
    Then 
\begin{equation}\label{f:main}
    \mathcal{D} [A] \le \exp (O(\mathcal{E}^{1/3}_*[A] \cdot \log^2  \mathcal{E}_*[A] )) \,.
\end{equation}
\label{t:main}
\end{theorem}
\begin{proof}
    Let $M = \mathcal{E}_* [A]$.
    Split the set $A$ into two sets $|A_1| \sim |A_2| \sim  |A|/2$ in an arbitrary way. Then one has $\E(A)\gg |A|^3/M$. 
    Applying 
    a version of the Balog--Szemer\'edi--Gowers Theorem (see, e.g., \cite{Schoen_BSzG}), we find $A'\subseteq A$, $|A'| \gg |A|/M$ such that $|A'-A'| \ll M^4 |A'|$. 
    Let $T$ be a parameter, $T=CM^2$, where $C>1$ is a sufficiently large absolute constant. 
    Using 
    some kind 
    of greedy algorithm, we construct a set $Z\subseteq A$, $Z\cap A' = \emptyset$ such that for any $x\in \Gr$ one has 
\begin{equation}\label{f:Z_property}
    |Z \cap (A'+x)| \le \frac{|A'|}{T} \,.
\end{equation}
    Indeed, at the first step of our algorithm put $Z_0 = A\setminus A' := (A')^c$ and if \eqref{f:Z_property} takes place for $Z=Z_0$, then we are done. 
    Otherwise, there is $x_1 \in \Gr$ such that $|(A')^c \cap (A'+x_1)| \ge |A'|/T$. 
    Put $Y_1 := A' \bigsqcup ((A')^c \cap (A'+x_1))$ and $Z_1 = Y^c_1$. 
    If \eqref{f:Z_property} holds for $Z=Z_1$, then we are done.
    If not, then there is $x_2 \in \Gr$ such that 
    $|Z_1 \cap (A'+x_2)| \ge |A'|/T$ and we put 
    $Y_2 = Y_1 \bigsqcup (Z_1 \cap (A'+x_2))$ and $Z_2 = Y^c_2$. 
    And so on. 
    Clearly, the algorithm must stop after at most $s_1:=[T|A|/|A'|] = O(M^3)$ number of steps.

    Now let $Y:=Y_{s_1}$, $Z=Z_{s_1}$ and $Y_* = Y\setminus A'$. 
    Below we assume that $|Z| \ge |A|/2$. 
    Since $Y \bigsqcup Z = A$, then thanks to our choice of parameter $T$ (here it is sufficient to have $T=\Omega(M)$) we 
    derive 
    that
\begin{equation}\label{tmp:21.02_1}
    \frac{|A'| |Y| |Z|^2}{M|A|} 
    \le 
    \frac{|Y|^2 |Z|^2}{M|A|} \le \E(Y_* \bigsqcup A',Z) \le 
    \E(Y_*,Z) + 2 \E(Y,A',Z,Z) 
\end{equation}
\begin{equation}\label{tmp:21.02_2}
    \le \E(Y_*,Z) + \frac{2 |A'||Y||Z|}{T} 
    \le \E(Y_*,Z) + \frac{|Y|^2 |Z|^2}{2M|A|} \,.
\end{equation}
    Thus 
    writing $\E(Y_*,Z) = |Y_*|^2 |Z|/T_*$ for some $T_* \ge 1$, we obtain 
\begin{equation}\label{tmp:21.02_3}
    \frac{|Y|^2 |Z|^2}{2M|A|} \le \E(Y_*,Z)  = 
    \frac{|Y_*|^2 |Z|}{T_*}
    \,,
\end{equation}
    and hence 
\begin{equation}\label{f:Y*_growth}
    |Y_*| \ge \sqrt{\frac{T_*}{M}} \cdot \frac{|Y|}{2} \ge \frac{|A'|}{2\sqrt{M}} \gg \frac{|A|}{M^{3/2}} \,.
\end{equation}
    Also, notice that estimate \eqref{tmp:21.02_3} implies the following rough bound on $T_*$ 
\begin{equation}\label{f:T_*_M}
    T_* \le \frac{2M |A|}{|Z|} \le 4M \,.
\end{equation}
    On the other hand, by the definition of  the quantity $T_*$ one has $\E(Y_*,Z) = \frac{|Y_*|^2 |Z|}{T_*}$ and hence applying Lemma \ref{l:E(A,B)_Schoen}, bound \eqref{f:Y*_growth}, combining with estimate \eqref{f:T_*_M}, we find $X\subseteq \Gr$ with 
\begin{equation}\label{f:X_second_case}
    |X| \le \min\left\{ \frac{4M|A|}{|Y_*|}, \frac{2T_* |Z|}{|Y_*|} \right\} \le
    \frac{4 \sqrt{M T_*} |Z|}{|Y|}
    \le 
    \frac{4 \sqrt{M T_*} |A|}{|A'|}
    \ll M^2  
\end{equation}    
    and such that 
\begin{equation}\label{f:T*_2}
    |Z \cap (Y_* + X)| \ge \frac{|Z|}{4T_*} \,.
\end{equation}
    Now we apply the same algorithm starting with the set $Y(1):=Y$.
    And so on. 
    Thus we construct an increasing sequence of sets $A' \subseteq Y(1) \subseteq Y(2) \subseteq \dots \subseteq Y(t') \subseteq A$ and let us put $Z(j) = Y^c(j)$. 
    Inequality \eqref{f:Y*_growth}, combined with the estimate $|A'| \gg |A|/M$ show that we obtain $|Z(t')| < |A|/2$ after at most $t' = O(M^{1/2} \log M)$ number of steps. 
    Let us improve the last  bound and check 
    that in view of inequality \eqref{f:T*_2} 
    the real number of steps $t'$ is $O(M^{1/3} \log M)$.
    Indeed put $R = [M^{1/3}]$ and consider the first $R$ sets $Y(j)$, $j\in [R]$. 
    Besides sets $Y(j)$ for any $j\in [R]$ we have sets $Z(j)$, where $Y(j) \bigsqcup Z(j) = A$ and the parameter $T_* (j)$. 
    Then we have two cases: either there are at least $R/2$ numbers $j\in [R]$ such that $T_* (j) \ge R$ or there exist  at least $R/2$ numbers $j\in [R]$ such that $T_* (j) < R$.  
    In the first case we use inequality \eqref{f:Y*_growth} and see that 
\begin{equation}\label{tmp:25.02_1}
    |Y(R)| \ge |A'| \left( 1 + \frac{\sqrt{R}}{2\sqrt{M}} \right)^{R/2} \ge  |A'| (1+c) \,,
\end{equation}
    where $c>0$ is an absolute constant. 
    In the second case one can apply estimate \eqref{f:T*_2} and obtain 
\begin{equation}\label{tmp:25.02_2}
    |Z(R)| \le |Z| \left( 1- \frac{1}{4R} \right)^{R/2}
    \le (1-\tilde{c}) |Z| \,,
\end{equation}
    where $\tilde{c}>0$ is another absolute constant and hence, say, 
    $$
        |Y(2R)| \ge (1+2^{-1} \tilde{c}) |Y|
        \ge (1+2^{-1} \tilde{c}) |A'| \,.
    $$
    Formulae \eqref{tmp:25.02_1}, \eqref{tmp:25.02_2} show that in any case we cannot have more than $t'=O(R \log M) = O(M^{1/3} \log M)$ steps.

    Having reached the inequality $|Z(t')|< |A|/2$, we apply the same algorithm to the set $Y(t')$ to obtain another increasing sequence of sets $Y(t') \subseteq Y(t'+1) \subseteq \dots \subseteq Y(t'+t'')$ and so on, but this time suppose that for all $j\ge t'$ one $|Z(j)| \ge \eps |A|$ holds, where $\eps:= (2^6 M)^{-1}$.
    Now we have $|Y(j)| \ge |A|/2$, $j\ge t'$, $T=C M^2$ and repeating the computations in \eqref{tmp:21.02_1}---\eqref{tmp:21.02_3}, we see that 
\[
    \frac{|Z (j)|^2 |A|}{8M} \le \frac{|Y (j)|^2 |Z(j)|^2}{2M|A|} \le \E(Y_* (j),Z(j))  = \frac{|Y_* (j)|^2 |Z(j)|}{T_* (j)} \,,
\]
    and hence 
\begin{equation}\label{f:Y*_growth2}
    |Y_* (j)| \ge |Z(j)| \cdot \left( \frac{T_* (j) |A|}{8M |Z(j)|} \right)^{1/2} \ge |Z(j)| \cdot \left( \frac{T_* (j)}{4M} \right)^{1/2}\,.
\end{equation}
    As above we consider two cases: either there are at least $R/2$ numbers $j\in [R]$ such that $T_* (j) \ge R$ or there exist  at least $R/2$ numbers $j\in [R]$ such that $T_* (j) < R$. 
    Thus  
    bounds \eqref{tmp:25.02_2}, \eqref{f:Y*_growth2}  show that 
    $t''=O(R \log M) = O(M^{1/3} \log M)$ 
    and therefore 
    the set $\mathcal{A}:= Y(t)$ has size at least $(1-\eps)|A|$, where $t:=t'+t'' \ll M^{1/3} \log M$.
    Let us show that $\mathcal{A}$ has small doubling and then apply Proposition \ref{p:A'_large} to obtain that our initial set $A$ has small doubling. 
    We have $\mathcal{A} \subseteq A' + X_1 + \dots + X_t$, where one has $|X_j| \ll M^{3}$ thanks to our choice of the parameter $T$,  bounds 
    \eqref{f:T_*_M}, \eqref{f:X_second_case}, \eqref{f:Y*_growth2} and the estimate $|A'| \gg |A|/M$. 
    It follows that 
\[
    |\mathcal{A} + \mathcal{A}| \le  (|X_1| \dots |X_t|)^2 |A'+A'| 
    \ll 
     \exp( O(t \log M) ) |A'| \ll  \exp( O(M^{1/3} \log^2 M) ) \cdot |\mathcal{A}| \,.
\]
    Using Proposition \ref{p:A'_large}, we complete the proof. 
$\hfill\Box$
\end{proof}

\section{A counterexample}
\label{sec:counterexm}

In this section we obtain the second part of Theorem \ref{t:main_intr}. 
Our construction follows  naturally from the proof of Theorem \ref{t:main} of  the previous section 
(in the notation of the latter theorem we basically construct the sets $Y(j)$ in the example below)
and 
resembles a ``bridge graph'' see, e.g., \cite[Section 6.5]{TV} and especially \cite{ER_connectedness} (note also that in \cite[Remark 4.3]{ER_connectedness} the connectedness $\a$ is of order $1/k^2$, not $1/k$ as claimed). 

\begin{theorem}
    Let $p>2$ be a prime number.
    Then there exists a set $A\subseteq \F_p^n$ such that 
\begin{equation}\label{f:constr}
    \mathcal{D}[A] = \min \left\{ \Omega_p \left( \frac{|A| \log^2 \log |A|}{\log^2 |A|} \right),\,  \exp(\Omega_p (\mathcal{E}^{1/4}_*[A] \cdot \log \mathcal{E}_*[A] )) \right\} \,.
\end{equation}
\label{t:constr}
\end{theorem}
\begin{proof}
    Let $a$, $M> p$ and $k$, $M\ll k \le M/2$, $k M^{k-1} \ll a$ be some integer parameters which we will choose later and we suppose that 
    $a/k$ 
    and $M+1$ are some  powers of 
    $p$. 
    We take $L_k < L_{k-1} < \dots < L_1 < \F_p^n$ be some subspaces, 
    $\frac{a}{kM^{j-1}} \le |L_j| \le \frac{ap}{kM^{j-1}}$, $j\in [k]$, and
    $H_2,\dots,H_k < \F_p^n$ be another collection of subspaces, where $|H_j|=M+1$, $j=2,\dots,k$.
    Put $H^*_j = H_j \setminus \{0\}$.
    We assume that $H^*_2,\dots,H^*_k$ and $L_1$ are mutually independent in the sense that 
    the sumset $H^*_2 + \dots + H^*_k + L_1$  is direct 
    (for example, $L_1$ occupies the first $m_1$ coordinates in the standard basis and each $H^*_2, \dots, H^*_k$ occupies some other coordinates in $\F^n_p$).
    Put $A_1 = L_1$ and $A_j= L_j \dotplus (H^*_2 \dotplus \dots \dotplus H^*_j)$, $j=2,\dots,k$ and let $A=\bigsqcup_{j=1}^k A_j$. 
    Then $a\le |A| \le p a$ and 
\begin{equation}\label{tmp:19.02_1}
    |A+A| \ge |A_1 + A_k| = |L_1| |H^*_2 + \dots + H^*_k|
    \ge \frac{M^{k-1} a}{k} 
    \gg_p |A| \cdot \exp(\Omega (M \log M))
    \,,
\end{equation}
    and thus $\mathcal{D}[A] = \exp(\Omega_p (M \log M))$.

    Now let us estimate $\mathcal{E}_*[A]$.
    Take any sets $S,T$ such that $A=S\bigsqcup T$. 
    We have either 
\begin{equation}\label{tmp:19.02_2}
    |S\cap A_1| \ge \frac{|A_1|}{2} = \frac{a}{2k} 
\end{equation}
    or $|T\cap A_1| \ge |A_1|/2$ (or both). 
    Without loosing of the generality assume that 
    inequality \eqref{tmp:19.02_2} 
    takes place. 
    Choose the maximal $j$ such that $|S\cap A_i| \ge |A_i|/2$ for all $i\in [j]$.
    Therefore  
    we see that if  $j<k$,
    then 
    $|T\cap A_{j+1}| \ge |A_{j+1}|/2$.
    Put $S' = S'_j := S \cap A_j$ and $T' = T'_j := T \cap A_{j+1}$. 
    By the construction one has 
\begin{equation}\label{tmp:20.02_1}
    |S'| \ge \frac{a}{2k} \ge \frac{|S|}{2kp} \,,
        \quad \quad 
        \mbox{ and }
        \quad \quad 
    |T'| \ge \frac{a}{2k} \ge \frac{|T|}{2kp} \,,
\end{equation}
    provided $j<k$. 
    Put 
    $\Lambda_S = H^*_2 \dotplus \dots \dotplus H^*_j$, $\Lambda_T = H^*_2 \dotplus \dots \dotplus H^*_{j+1}$ and then  by  the mutual independence, we have  
\[
    S' = \bigsqcup_{\la \in \Lambda_S} \{ s'\in S' ~:~ s' = \lambda + l,\, l\in L_j \}
    = \bigsqcup_{\la \in \Lambda_S} S'_\la \,,
\]
    and similarly for $T'$.
    Let $\tilde{\Lambda}_S$ be the collection of $\la \in \Lambda_S$ such that $|S'_\la| \ge |S'|/(2|\Lambda_S|)$. 
    Then
\[
    \sum_{\la \in \tilde{\Lambda}_S} |S'_\la| \ge \frac{|S'|}{2} \ge \frac{|L_j| M^{j-1}}{4}
\]
    and hence $|\tilde{\Lambda}_S| \ge |\Lambda_S|/4$.
    In a similar way, one can define $\tilde{\Lambda}_T$ and then 
    $|\tilde{\Lambda}_T| \ge |\Lambda_T|/4$.
    Finally, put $S'' = \bigsqcup_{\la \in \tilde{\Lambda}_S} S'_\la$ and 
    $T'' = \bigsqcup_{\la \in \tilde{\Lambda}_T} T'_\la$.
    Let us 
    estimate 
    $\E(S'',T'')$ 
    from below, this will give us a lower bound for $\E(S,T)$. 
    Applying the mutual independence, the estimate $j\le k\le M/2$ and the Cauchy--Schwarz inequality twice, we  have 
\begin{equation}\label{f_E_comp1}
     \E(S'', A_j, T'',A_{j+1})  \ge  \sum_{\la_1, \la_2 \in \tilde{\Lambda}_S,\,  
     \mu_1, \mu_2 \in \tilde{\Lambda}_T ~:~ \la_1 + \la_2 = \mu_1 + \mu_2}
     \E(S'_{\la_1}, L_j, T'_{\mu_1}, L_{j+1}) 
\end{equation}
\[
    \ge 
     \sum_{\la_1, \la_2 \in \tilde{\Lambda}_S,\,  
     \mu_1, \mu_2 \in \tilde{\Lambda}_T ~:~ \la_1 + \la_2 = \mu_1 + \mu_2}
     \frac{|S'_{\la_1}| |T'_{\mu_1}||L_j||L_{j+1}|}{|L_{j}|}
        \gg 
        \frac{|L_{j+1}| |A_j| |A_{j+1}|}{|\Lambda_S||\Lambda_T|} \cdot \E(\tilde{\Lambda}_S, \tilde{\Lambda}_T) 
\]
\[
    \gg \frac{|L_{j+1}| |A_j| |A_{j+1}| |\Lambda_S| |\Lambda_T|}{|\Lambda_S + \Lambda_T|}
    \ge 
    \left( 1 - \frac{1}{M+1} \right)^j \cdot |L_{j+1}| |A_j| |A_{j+1}| |\Lambda_S|
\]
\begin{equation}\label{f_E_comp2}
    \gg
    \frac{|L_{j+1}| |S'|^2 |T'|}{|L_j|} 
    \gg_p 
    \frac{|S'|^2 |T'|}{M} \,. 
\end{equation}
    Now 
    using 
    the Cauchy--Schwartz inequality one more time, 
    we get 
\begin{equation}\label{tmp:27.02_1}
    \E(S,T) \E(A_j, A_{j+1}) \ge \E(S'',T'') \E(A_j, A_{j+1}) \gg_p  \frac{|S'|^4 |T'|^2}{M^2} \,.
\end{equation}
    From  the mutual independence it is easy to see that 
\begin{equation}\label{f:E_Aj_Aj+}
    \E(A_j,A_{j+1}) \ll_p \left( \frac{a}{k} \right)^3 M^{-(j-1)} M^{-2j} (M^{j-1})^3 M \ll \frac{a^3}{Mk^3} \,,
\end{equation}
    similarly,
\begin{equation}\label{f:E_Aj_Aj}
     \E(A_j) \ll_p \left( \frac{a}{k} \right)^3 M^{-3(j-1)}  (M^{j-1})^3 \ll \frac{a^3}{k^3} \,,
\end{equation}    
    and hence in view of \eqref{tmp:20.02_1} and our choice of  the parameter $k$ one derive from \eqref{tmp:27.02_1}
\begin{equation}\label{f:E_first}
    \E(S,T) \gg_p \frac{k^3 |S'|^4 |T'|^2}{Ma^3} \gg \frac{k |S'|^2 |T'|^2}{M a}
    \gg \frac{|S|^2 |T|^2}{Mk^3 a} \gg \frac{|S|^2 |T|^2}{M^4 |A|} \,.
\end{equation}

    Now assume that $j=k$.
    In this case we have $|S\cap A_i| \ge |A_i|/2$ for all $i\in [k]$ and in particular $|S| \ge  a/2$.
    Let $J\subseteq [k]$ be the collection of indexes such that $|T\cap A_{j}| \ge |T|/(2k)$. 
    For $j\in J$ put $T^*_{j} := T\cap A_{j}$. 
    Applying the previous argument for all $j\in J$ and repeating the computations in \eqref{f_E_comp1}---\eqref{f_E_comp2}, we obtain 
\[
    \E(S''_j,A_{j}, T^*_{j}, A_{j})
    \gg |L_{j}| |A_{j}| |T^*_{j}| |\Lambda_S|
    \gg 
    |S'_{j}|^2 |T^*_{j}| \,.
\]
    Thus summing over $j\in J$ and using the Cauchy--Schwarz inequality,  we obtain as above
\[
    \sum_{j\in J} \E^{1/2} (S''_j,T^*_{j})  \E^{1/2} (A_j)
    \gg 
    \sum_{j\in J} 
    |S'_{j}|^2 |T^*_{j}|
    \gg 
    \frac{|S|^2 |T|}{k^2} \,.
\]
    It remains to apply  the Cauchy--Schwarz inequality one more time, as well as bound \eqref{f:E_Aj_Aj} to derive 
\[
    \frac{a^3}{k^3} \E(S,T) k \gg_p \frac{|S|^4 |T|^2}{k^4} \gg  \frac{|S|^2 |T|^2 a^2}{k^4} \,.
\]
    This  implies 
\begin{equation}\label{f:E_second}
    \E(S,T) \gg_p \frac{|S|^2 |T|^2 }{k^2 a} \gg 
    \frac{|S|^2 |T|^2 }{M^2 |A|}  \,.
\end{equation}
    Combining bounds \eqref{f:E_first} and \eqref{f:E_second}, we see that $\mathcal{E}_* [A] \ll_p M^4$. 
    Substituting this bound into \eqref{tmp:19.02_1}, we obtain the result. 
    It remains to check the condition $k M^{k-1} \ll M^M \ll a$ but if not, then choose $k\sim M$ such that $M^M \sim |A|$ and then bound \eqref{tmp:19.02_1} gives us $\mathcal{D}[A] \gg \Omega_p \left( \frac{|A| \log^2 \log |A|}{\log^2 |A|} \right)$. 
This completes the proof.
$\hfill\Box$
\end{proof}

\section{An application}
\label{sec:application}

Now we obtain an application to the arithmetic regularity lemma, see \cite{GS_max_AP3}. We formulate our result in the following form.

\begin{theorem}
    Let $A \subseteq \Gr$ be  a set, $\E(A) = |A|^3/K$ and let $\eps, \o \in (0,1/4]$ be parameters.
    Then there is a decomposition of $A$ as a disjoint union
    $A=\left(\bigsqcup_{j=1}^k A_j \right) \bigsqcup \Omega$ such that\\
    $(i)~$ (Components are large). $|A_j| \ge \sqrt{\o/2K} \cdot |A|$, $j\in [k]$.\\
    $(ii)~$ (Components are structured). $\mathcal{E}_* [A_j] \le 4 K (\o \eps)^{-2}$, $j\in [k]$.
    In particular, for any $j\in [k]$ one has $\mathcal{D}[A_j] \ll \exp ( O (K^{1/3} (\o \eps)^{-2/3} \cdot \log^2 (K\o^{-1} \eps^{-1}) ) )$.\\
    $(iii)~$ (Different components do not communicate). $ \E(A_i,A_j)\le \eps \frac{|A_i|^2 |A_j|^2}{|A|}$, $i,j \in [k]$, $i\neq j$. \\
    $(iv)~$ (Noise term). $\E (\Omega,A) \le \omega \E(A)$, $j\in [k]$.
\label{t:str}
\end{theorem}
\begin{proof}
    In the proof we follow the argument from \cite{ER_connectedness}. 
    Put $\eps_1= \o \eps/2$. 
    For $B\subseteq A$ write $\mu (B) = |B|/|A|$.
    Consider all partitions of $A$ into disjoint sets $A_1,\dots, A_l$, $l\ge 1$ and choose 
    one such that the sum 
\[
    \sum_{1\le i<j\le l} (\E(A_i,A_j, A,A) - \eps_1 \mu (A_i) \mu (A_j) \E(A))
\]
    is minimal. 
    If the minimal value is attained at several partitions, take any of them. 
    Clearly, for an arbitrary $i\in [l]$ and any $S,T$ such that $S\bigsqcup T = A_i$ one has 
    $\E(S,T, A,A) \ge \eps_1 \mu (S) \mu (T) \E(A)$ (otherwise we have a  contradiction with the minimality).
    Thus by the Cauchy--Schwarz inequality 
    the following holds 
\[
    \E(S,T) \ge \eps^2_1 \frac{|S|^2 |T|^2 \E(A)}{|A|^4} = 
    \eps^2_1 \frac{|S|^2 |T|^2 }{K|A|} 
\]  
    and therefore  $\mathcal{E}_* [A_i] \le 4 K (\o \eps)^{-2}$ for all $i\in [l]$. 
    Then Theorem \ref{t:main} gives us the required upper bound for $\mathcal{D}[A_i]$. 
    Using the minimality again, we see that for all $i,j\in [l]$, $i\neq j$ the following holds
\begin{equation}\label{f:DN_communicate}
    \E(A_i,A_j, A,A) \le \eps_1 \mu (A_i) \mu (A_j) \E(A) \,.
\end{equation}
    Putting $I:= \{ i ~:~ |A_i| \le \sqrt{\o/2K} \cdot  |A| \}$ and $\Omega = \bigsqcup_{i\in I} A_i$, we obtain 
\[
    \E(A,\Omega) = \sum_{i,j \in [l]} \E(A_i,A_j, \Omega,\Omega)
    \le 
    \sum_{m \in [l]} \E(A_m,\Omega) + \eps_1 \E(A) \sum_{i,j \in [l]} \mu (A_i) \mu (A_j) 
\]
\[
    \le 
    \sum_{m \in [l]}  \sum_{i \in I} \E(A_m,A_m, A_{i}, A)
    + \eps_1 \E(A) 
    \le 
    \sum_{i\in I} \E(A_i) + 3 \eps_1 \E(A) 
\]
\[
    \le 
    2^{-1} \o |A|^3/K + 3 \eps_1 \E(A)
    \le 
    \o \E(A) 
\]
and thus we have obtained (iv). 
    Finally, in view of \eqref{f:DN_communicate} for any $i,j\notin I$, $i\neq j$ one has 
\[
    \E(A_i,A_j) \le \E(A_i,A_j,A,A) \le \eps_1 \mu (A_i) \mu (A_j) \E(A)
    \le 
    \eps \frac{|A_i|^2 |A_j|^2}{|A|} \,.
\]
This completes the proof.
$\hfill\Box$
\end{proof}

\bp

Our dependence on the parameters is 
better than in \cite[Proposition 3.2]{GS_max_AP3} and 
in \cite[Theorems 4.1, 4.4, 4.6]{ER_connectedness}. 
Although  the proof of Theorem \ref{t:str} follows  the method from \cite{ER_connectedness}
the advantage is that we use the more subtle quantity $\mathcal{E}_* [A]$ in our application of Theorem \ref{t:main}. 
Finally, we note that sometimes in the formulation of Theorem \ref{t:str} there is an additional parameter $L$, but this twist is completely unimportant and is a consequence of the result given above (for more details see \cite{GS_max_AP3}).


\bibliographystyle{abbrv}

\bibliography{bibliography}{}

\end{document}